\documentclass[12pt]{article}
\usepackage{verbatim}
\usepackage{amssymb}
\usepackage{amsthm}
\usepackage{amsmath}
\usepackage{mathrsfs}
\usepackage[margin=1in]{geometry}

\let\cal=\mathcal      

\def\mcc{M\raise.5ex\hbox{c}C}
\def\mccarthy{M\raise.5ex\hbox{c}Carthy}
\def\Hu{\H}
\def\sz{Szeg\Hu{o} }

\def\K{{\cal K}}

\def\l{\lambda}

\def\vare{\varepsilon}

\let\i=\infty

\def\la{\langle}
\def\ra{\rangle}
\def\={\ = \ }

\def\A{{\cal A}}
\def\BB{{\cal B}}

\def\C{\mathbb C}

\def\D{\mathbb D}

\def\NN{\mathbb N}

\def\be{\setcounter{equation}{\value{theorem}} \begin{equation}}
\def\ee{\end{equation} \addtocounter{theorem}{1}}
\def\beq{\begin{eqnarray*}}
\def\eeq{\end{eqnarray*}}

\def\vs{\vskip 5pt}

\def\bp{{\sc Proof: }}
\def\ep{{}{\hfill $\Box$} \vskip 5pt \par}

\def\bl{\begin{lemma}}
\def\el{\end{lemma}}
\def\bt{\begin{theorem}}
\def\et{\end{theorem}}
\def\bprop{\begin{prop}}
\def\eprop{\end{prop}}
\def\bd{\begin{definition}}
\def\ed{\end{definition}}
\def\br{\begin{remark}}
\def\er{\end{remark}}
\def\bexer{\begin{exercise}}
\def\eexer{\end{exercise}}
\def\bfig{\begin{figure}}
\def\efig{\end{figure}}

\usepackage[usenames,dvipsnames] {color}

\definecolor{dark_purple}{rgb}{0.4, 0.0, 0.4}
\definecolor{dark_green}{rgb}{0.0, 0.7, 0.0}

\numberwithin{equation}{section}

\title{The Hardy-Weyl algebra}
\author{Jim Agler
and
John E. M\raise.5ex\hbox{c}Carthy
\thanks{Partially supported by National Science Foundation Grant  
DMS 2054199}
}
\date{\today}

\newcommand{\mc}{M\raise.45ex\hbox{c}Carthy}

\def\ip#1#2{\langle#1,#2\rangle}

\usepackage{mathtools}

\renewcommand\NN{{\mathbb N}}

\newcommand\wt{\widehat{T}}

\renewcommand\L{\Lambda}
\newcommand{\Kz}{\K_{\A}}
\newcommand{\Az}{\A_0}
\renewcommand\P{{\mathcal P}}
\newcommand\Q{{\mathcal Q}}
\newcommand\cpd{C(\partial \D(1,1))}
\newcommand\hfp{H_{f_+}}

\newcommand\fno{f_{n+1}}
\newcommand\Fno{F_{n+1}}
\newcommand\fz{f_0}
\newcommand\fo{f_1}
\newcommand\Fo{F_1}

\def\d{\mathbb{D}}
\def\t{\mathbb{T}}

\def\calc{\mathcal{C}}
\def\calk{\mathcal{K}}

\def\ltwo{L^2}
\def\set#1#2{\{ #1 \, | \, #2\}}
\def\norm#1{\| #1 \|}
\newcommand{\twopartdef}[4]
{
	\left\{
		\begin{array}{ll}
			#1 & \mbox{if } #2 \\
			#3 & \mbox{if } #4
		\end{array}
	\right.
}
\theoremstyle{definition}
\newtheorem{defin}[equation]{Definition}
\newtheorem{definition}[equation]{Definition}
\newtheorem{lem}[equation]{Lemma}
\newtheorem{lemma}[equation]{Lemma}
\newtheorem{prop}[equation]{Proposition}
\newtheorem{thm}[equation]{Theorem}
\newtheorem{claim}[equation]{Claim}

\newtheorem{question}[equation]{Question}

\newtheorem{cor}[equation]{Corollary}
\newtheorem{rem}[equation]{Remark}
\renewcommand\bt{\begin{thm}}
\renewcommand\et{\end{thm}}
\renewcommand\be{\begin{equation}}
\renewcommand\ee{\end{equation}}
\begin{document}

\bibliographystyle{plain}

\maketitle

\begin{abstract}
We study the algebra $\A$ generated by the Hardy operator $H$ and the operator $M_x$ of multiplication by $x$ on $L^2[0,1]$. We call $\A$ the Hardy-Weyl algebra.
We show that its quotient by the compact operators is isomorphic to the algebra of functions that are continuous on $\Lambda$ and analytic on the interior of $\Lambda$ for a planar set $\Lambda$ = $[-1,0] \cup \overline{ \D(1,1)}$, which we call the lollipop. We find a Toeplitz-like short exact sequence for the $C^*$-algebra generated by $\A$.

We study the operator $Z = H - M_x$, show that its point spectrum is $(-1,0] \cup \D(1,1)$,  and that the eigenvalues grow in multiplicity as the points move to $0$ from the left. 
\end{abstract}

\section{Introduction}

The classical Weyl algebra is generated by the operators of 
multiplication by $x$, denoted $M_x$, and differentiation, denoted by $D$. These operators satisfy the commutator relation
\be
\label{eqwc}
DM_x - M_x D \= 1 .
\ee
The algebra generated by these relations has been studied extensively in both algebra and operator theory---see e.g. 
the books \cite{os12, masch16, sch20,  gal15}.
In operator theory, the study is complicated by the fact that no bounded operators satisfy \eqref{eqwc}.
In this note, we shall study the associated algebras that arise when one replaces the differentiation operator  $D$ by
a bounded integration operator. 

 The  Hardy operator $H$ is defined on $L^2[0,1]$  by 
\[
H f (x) \=  \frac{1}{x} \int_0^x f(t) dt .
\]
Let $V = M_x H $ denote the Volterra operator
\[
V : f \mapsto  \int_0^x f(t) dt ,
\]
which is a right inverse of $D$.
These operators give rise to   a new set of relations instead of \eqref{eqwc}, namely
\beq
VM_x - M_x V &\=& - V^2 \\
HM_x - M_x H &=& - HM_x H  \= - HV.
\eeq

 We shall use $L^2$ to mean $L^2[0,1]$ throughout.
\begin{definition}
{\rm
Let $\A$ denote the closure of the  unital algebra generated by $H$ and $M_x$ in the norm topology of $B(L^2)$.
We call $\A$ the Hardy-Weyl algebra.
}\end{definition}

When is an operator in $\A$? Can all the elements be described?
Since the commutator of $H$ and $M_x$ is compact, we first describe
 the quotient of $\A$ by the compact operators. Define $\L$, a subset of the plane, by
\[
\L \= [-1,0] \cup \overline{ \D(1,1)} .
\]
We call $\L$ the {\em lollipop}. The {\em lollipop algebra} $A(\L)$ is the Banach algebra of functions that are continuous on $\L$ and analytic on ${\rm int}(\L)$,
equipped with the maximum modulus norm.
 Let $\K$ denote the compact operators on $L^2$,
and let $\Kz = \K \cap \A$,
\bt
\label{thma1}
There is a Banach algebra isomorphism $\gamma$ from $A(\L)$ onto $\A/\Kz$.
It is given by 
\[
\gamma(f) \= f|_{[-1,0]} (-M_x) + f|_{\overline{\D(1,1)}} (H) - f(0) .
\]
\et
Let $\theta : \A \to A(\L)$ be defined by $\theta (T) = \gamma^{-1}([T])$, where $[T]$ is the projection
of $T$ onto $\A/\Kz$.
As a corollary to Theorem \ref{thma1} we obtain that every element $T$ of $\A$ can be written uniquely as
\be
\label{eqa11}
T \= M_\phi + g(H) + K ,
\ee
where $\phi \in C([0,1])$, $g$ is in $A(\overline{\D(1,1)})$, $\phi(0) = g(0)$, and $K \in \Kz$.

\vs
If we look at $C^*(\A)$, the $C^*$-algebra generated by $\A$, we get something similar to the Toeplitz algebra short exact sequence.
See \cite{brwi19, ea20, rrs19} for some recent results on the Toeplitz algebra, and  \cite{arme20, ha18, pa90, ps91}
for some applications.
 The lollipop algebra is replaced by the functions that are continuous on the boundary.
\bt
\label{thma2}
There is a short exact sequence of $C^*$-algebras
\[
0 \to\K    \to C^*(\A)  \to C(\partial \L)  \to 0 .
\]
\et

\vs
In addition to studying the algebra generated by $H$ and $M_x$, one may also study the smaller algebra generated
by $V$ and $M_x$. 
 \begin{definition}
 \label{defina3}
{\rm The algebra $\Az$ is the norm-closed unital algebra generated by $V$ and $M_x$.
}\end{definition}
One can see that $\Az$ is a proper subalgebra of $\A$ by noting that its quotient by the compact operators
is isomorphic to $C([0,1])$.

The lattice of closed invariant subspaces of  the Volterra operator $V$ was shown independently by Brodskii and Donoghue \cite{br57,don57}
to be
\[
{\rm Lat}(V) \= \Big\{ \{ f \in L^2 : f = 0 {\rm\ on\ } [0,s] \} : s \in [0,1] \Big\}.
\]
We let  ${\rm AlgLat}(V)$ denote the set of bounded operators on $L^2$ that leave invariant
every element of ${\rm Lat V}$. This is a large algebra---it includes all right translation operators for example.
It is described by the following theorem of Radjavi and Rosenthal  \cite[Example 9.26]{raro73}. 
\bt
\label{thmrr}
The algebra ${\rm AlgLat}(V)$ is the weak operator topology closure of the algebra generated
by $V$ and $M_x$.
\et
It follows from Theorem \ref{thmrr} that $H$ is in the WOT closure of $\A_0$, and hence $\A$ and $\A_0$
have the same WOT closure. 
However, no extra compact operators are added in the WOT closure.


 \bt
 \label{thma3}
 \[
 {\rm AlgLat}(V) \cap \K \ \subset \ \Az .
 \]
 \et
 
 \vs
In Section \ref{secg} we consider the operator $Z = H - M_x$.  
It follows from Theorem \ref{thma1}  that $[Z]$ generates $\A/\Kz$.
We show $Z$ has a surprisingly rich collection of eigenvectors.
\bt
\label{thma4}
Let $Z = H - M_x$. Then 
\[
\sigma_p(Z) \= (-1,0] \cup \D(1,1) .
\]
\et
The algebraic multiplicity of the eigenvalues of $Z$ on the stick $(-1,0]$ increases
as $\lambda \to 0^-$, and hence  operators $X$  in the closed algebra generated by $Z$ have the property that $\theta(X)$ is not just in $A(\L)$, but is smoother.
We prove:

\bt
\label{thma4}
Let $X$ be in the norm-closed algebra generated by $Z$. Then $\theta(X)$ is  $C^m$ 
  on $ (-\frac{2}{2m+1},0)$.
  \et

\section{Preliminaries}
\label{secpre}

\begin{definition}
A {\em monomial operator} is a bounded linear operator $T: L^2[0,1] \to L^2 [0,1]$ with the property  that there exist constants $c_n$ and $p_n$ so that
\be
\label{eqa1}
T : x^n \mapsto c_n \, x^{p_n} \qquad \forall n \in \NN.
\ee 
We call it a {\em flat monomial operator} if there exists some $\tau$ so that $p_n = n + \tau$
for all $n$.
\end{definition}

In \cite{amHI} we showed that every flat monomial operator is in ${\rm AlgLat}(V)$, and hence by Theorem
$\ref{thmrr}$ in the weak closure of $\Az$.
The Volterra operator $V$ is Hilbert-Schmidt, and hence compact. See e.g. \cite{lax02} for a proof.

Let $H^2$ denote the Hardy space of the unit disk, and $k_w(z) = \frac{1}{1-\bar w z}$ the \sz kernel.
We shall let $S : f(z) \mapsto z f(z)$ denote the unilateral shift on $H^2$.
Let 
\[
\beta(z) \= \frac{1}{2-z} ,
\]
and let $C_\beta : f \mapsto f \circ \beta$ denote the composition operator of composing with $\beta$.

There is a unitary $U : L^2 \to H^2$ that is defined on monomials by
\be
\label{eqb2}
U : x^s \mapsto \frac{1}{s+1} k_{\frac{\bar s}{\bar s +1}} ,
\ee
and extended by linearity and continuity to the whole space.
If $T \in B(L^2)$, we shall let $\wt$ denote $U T U^*$.
It is easy to see that $U$ is unitary, as it preserves inner products.
In  \cite{amMS} we prove that $U$ is given by the formula
\be
\label{eqb1}
U f (z) \= \frac{1}{1-z} \int_0^1 f(x) x^{\frac{z}{1-z}} dx ,
\ee
and show that
\beq
\widehat{M_x} &\=& S^* C_\beta^* \\
\widehat{V} &\=& (1 - S^*) C_\beta^* \\
\widehat{H} &=& 1 - S^* .
\eeq

If $X$ is a compact subset of $\C$, we shall let 
$C(X)$ denote the Banach algebra of functions that are continuous on $X$, with the maximum modulus norm.
We shall let 
$A(X)$ denote the subalgebra of functions that are continuous on $X$
and analytic on the interior of $X$, and $P(X)$ denote the closure of the polynomials in $C(X)$. 
A theorem of Mergelyan \cite{mer52} says that if the complement of $X$ is connected, then $A(X) = P(X)$.

\section{The Calkin Hardy-Weyl Algebra}
\label{secb}
 We let $\calk$ denote the ideal of compact operators acting on  $\ltwo$ and set
\[
\calk_0 = \A \cap \calk.
\]
Evidently, $\calk_0$ is a 2-sided ideal in $\A$. Consequently, we may define an algebra $\calc$, the \emph{Calkin Hardy-Weyl algebra}, by
\[
\calc = \A / \calk_0
\]
If $T \in \A$ we let $[T]$ denote the coset of $T$ in $\calc$, i.e.,
\[
[T] = \set{T+K}{ K\in \calk_0}.
\]
\begin{prop}\label{calk.prop.10}
$\calc$ is an abelian Banach algebra.
\end{prop}
\begin{proof}
That $\calc$ is a Banach algebra follows from the fact that $\calk_0$ is closed in $\A$. To see that $\calc$ is abelian, observe that as $M_x$ and $H$ generate $\A$,
$[M_x]$ and $[H]$ generate $\calc$. Furthermore, as $M_xH = V \in \calk_0$,
\be\label{calk.10}
[M_x][H] = 0.
\ee
Likewise, as $HM_x =(1-H)V \in \calk_0$,
\be\label{calk.20}
[H][M_x] =0,
\ee
so that in particular we have that
\[
[M_x][H]=[H][M_x].
\]
As $[M_x]$ and $[H]$ commute and generate $\calc$, $\calc$ is abelian.
\end{proof}

\subsection{A Uniform Algebra Homeomorphically Isomorphic to $\calc$}
\label{secc2}

We begin by defining an algebra by gluing together two simpler algebras whose maximal ideal spaces overlap at a single point. Let
\[
\P \=\{\ f=(f_-,f_+)\ :\ f_-\in C([-1,0]),\ f_+ \in A(\overline{\d(1,1)}) \text{ and } f_-(0)=f_+(0)\ \}
\]
where we view $\P$ as an algebra with the operations
\[
cf=(cf_-,cf_+),\ f+g=(f_-+g_-,f_++g_+), \text{and } fg=(f_-g_-,f_+g_+),
\]
and the norm
\[
\norm{f} = \max \big \{\max_{t\in [-1,0]} |f_-(t)|, \max_{z\in (\overline{\d(1,1)}}|f_+(z)|\ \big\}.
\]
We abuse notation by letting
\[
f(0)=f_-(0)
\]
 when $f\in \P$.

We note that if $f\in C([-1,0])$, then as $-M_x$ is self-adjoint and has spectrum equal to $[-1,0]$,  we may form the operator $f(-M_x)$. Likewise, as $H$ is cosubnormal and has spectrum equal to $\overline{\d(1,1)}$, if $g\in A(\overline{\d(1,1)})$, then we may form the operator $g(H)$. Concretely,
\[
f(-M_x) = M_{f(-x)}\ \text{ and }\ \widehat{g(H)}=M_h^*,
\]
where $h(z) = \overline{g(1-\bar z)}$, and $M_h$ denotes multiplication by $h$.
\begin{lem}\label{calk.lem.10}
If $f\in C([-1,0])$ and $g\in A(\overline{\d(1,1)})$, then
\[
[f(-M_x)][g(H)]=g(0)[f(-M_x)]+f(0)[g(H)]-f(0)g(0).
\]
\end{lem}
\begin{proof}
Since $[-1,0]$ is a spectral set for $-M_x$ and $\overline{\d(1,1)}$ is a spectral set for $H$ it suffices to prove the lemma in the special case when $f$ and $g$ are polynomials. Let $f(x)=f(0)+xf_1(x)$ and $g(x)=g(0)+xg_1(x)$. Using \eqref{calk.10} and \eqref{calk.20} we see that
\begin{align*}
[f(-M_x)][g(H)]&=\big([f(0)]+[-M_x][f_1(-M_x)]\big)\big([g(0)]+[H][g_1(H)]\big)\\  \\
&=f(0)g(0)+g(0)[-M_x][f_1(-M_x)]+f(0)[H][g_1(H)]\\  \\
&=f(0)g(0)+g(0)[f(-M_x)-f(0)]+f(0)[g(H)-g(0)]\\ \\
&=g(0)[f(-M_x)]+f(0)[g(H)]-f(0)g(0).
\end{align*}
\end{proof}
If $f \in \P$ we define $\gamma(f) \in \A$ by the formula
\[
\gamma(f) = f_-(-M_x) + f_+(H)-f(0).
\]
We also define $\Gamma:\P \to \calc$ by the formula
\[
\Gamma(f) = [\gamma(f)]
\]
\begin{prop}\label{calk.prop.20}
$\Gamma$ is a continuous unital homomorphism.
\end{prop}
\begin{proof}
$\gamma$ is linear and $\gamma(1)=1$. Therefore, $\Gamma$ is linear and $\Gamma(1)=1$. Also,
\begin{align*}
\norm{\Gamma(f)}&=\norm{[\gamma(f)]}\\
&\le \norm{\gamma(f)}\\
&=\norm{f_-(-M_x) + f_+(H)-f(0)}\\
&\le\norm{f_-(-M_x)} +\norm{f_+(H)}+|f(0)|\\
&=\max_{t\in [-1,0]}|f_-(t)|+\max_{z\in \overline{\d(1,1)}}|f_+(z)|+|f(0)|\\
&\le 3\norm{f},
\end{align*}
so $\Gamma$ is continuous.


Finally, to see that $\Gamma$ preserves products, fix $f,g \in \P$.
\begin{align*}
\Gamma(f)\Gamma(g)&=[\gamma(f)]\ [\gamma(g)]\\ \\
&=[f_-(-M_x)+f_+(H)-f(0)]\ [g_-(-M_x)+g_+(H)-g(0)]\\ \\
&=\Big([f_-(-M_x)][g_-(-M_x)]+[f_+(H)][g_+(H)]\Big)\\ \\
&\ \  +\Big([f_-(-M_x)][g_+(H)]+[g_-(-M_x)][f_+(H)]\Big)\\ \\
&\ \  -\Big(f(0)[g_-(-M_x)+g_+(H)]+g(0)[f_-(-M_x)+f_+(H)]\Big)\\ \\
&\ \ +f(0)g(0)\\ \\
&=\ \ A+B-C+f(0)g(0).
\end{align*}
But
\begin{align*}
A&=\Big([f_-(-M_x)][g_-(-M_x)]+[f_+(H)][g_+(H)]\Big)\\
&=\Big([f_-g_-(-M_x)]+[f_+g_+(H)]-f(0)g(0)\Big)+f(0)g(0)\\
&=[\gamma(fg)]+f(0)g(0)\\
&=\Gamma(fg)+f(0)g(0),
\end{align*}
and using Lemma \ref{calk.lem.10}, we see that
\begin{align*}
B&=\Big([f_-(-M_x)][g_+(H)]\Big)+\Big([g_-(-M_x)][f_+(H)]\Big)\\
&=\Big(g(0)[f_-(-M_x)]+[f(0)g_+(H)]-f(0)g(0)\Big)+
\Big(f(0)[g_-(-M_x)]+[g(0)f_+(H)]-f(0)g(0)\Big)\\
&=C-2f(0)g(0).
\end{align*}
Therefore,
\begin{align*}
\Gamma(f)\Gamma(g)&= A+B-C+f(0)g(0)\\
&=(\Gamma(fg)+f(0)g(0))+(C-2f(0)g(0))-C+f(0)g(0)\\
&=\Gamma(fg).
\end{align*}
\end{proof}
\begin{lem}\label{calk.lem.20}
If $p$ is a polynomial in two variables and we define $f\in \P$ by letting
\[
f_-(t) = p(t,0),\ t\in[-1,0]\ \ \ \ \text{and}\ \ \ \ f_+(z)=p(0,z),\ z \in \overline{\d(1,1)},
\]
then
\[
p([-M_x],[H]) = \Gamma(f).
\]
\end{lem}
\begin{proof}
If $p=p(x,y)$ is a polynomial in two variables and we let
\[
q(x,y)=p(x,y)-p(x,0)-p(0,y)+p(0,0),
\]
then
\[
p(x,y)=p(x,0)+p(0,y)-p(0,0) + q(x,y)
\]
and
\[
q([-M_x],[H])=0.
\]
Therefore,
\begin{align*}
p([-M_x],[H])&=p([-M_x],0)+p(0,[H])-p(0,0)\\
&=f_- ([M_x])+f_+([h])-f(0)\\
&=[\gamma(f)]\\
&=\Gamma(f).
\end{align*}
\end{proof}
\begin{cor}\label{calk.cor.10}
The range of 
$ \Gamma$ is dense in $\calc$.
\end{cor}
\begin{proof}
This follows immediately from Lemma \ref{calk.lem.20} by recalling that $[-M_x]$ and $[H]$ generate $\calc$ (cf. proof of Proposition \ref{calk.prop.10}).
\end{proof}
Lemma \ref{calk.lem.20} suggests that we consider the subset $\P_0$ of $\P$ defined by
\[
\P_0=\set{f\in \P}{f_- \text{ and } f_+ \text{ are polynomials}}.
\]
We note that it  follows from the facts that the polynomials are dense in both $C([-1,0])$ and $A(\overline{\d(1,1)})$ that $\P_0$ is dense in $\P$.
\begin{lem}\label{calk.lem.30}
If $s\in [-1,0]$, then
\be\label{calk.30}
|f_-(s)| \le \norm{\Gamma(f)}
\ee
for all $f\in \P$.
\end{lem}
\begin{proof}
As $f$ is continuous, it suffices to prove the lemma under the assumption that $s\in (-1,0)$. For $n$ satisfying $1/n < \min \{s, 1-s\}$ we define a unit vector $\chi_n \in \ltwo$ by the formula
\[
\chi_n(t)=
\twopartdef{\sqrt{\frac{n}{2}}}{|t-s|\le 1/n}{0}{|t-s|>1/n}
\]
We observe  that the mean value theorem for integrals implies that
\[
\lim_{n\to \infty}\ip{g\, \chi_n}{\chi_n} = g(s)
\]
whenever $g \in C([0,1])$. Also, as $\chi_n \to 0$ weakly,
\[
\lim_{n\to \infty} \norm{K\chi_n} =0
\]
whenever $K$ is a compact operator acting on $\ltwo$. In particular, as $V$ is compact and $V\chi_n(t)=0$ when $t\in [0,s-1/n)$,
\[
\lim_{n\to \infty} \norm{H\chi_n} =\lim_{n\to \infty} \norm{M_{1/x}V\chi_n} =0.
\]
More generally, if $q$ is a polynomial and $q(0)=0$, write $q(z) = z r(z)$, and we get
\[
\lim_{n\to \infty} \norm{q(H)\chi_n}=\lim_{n\to \infty} \norm{r(H)\ H\chi_n}=0.
\]

Now fix $f\in \P_0$ and a compact operator $K$ acting on $\ltwo$. Using the observations in the previous paragraph we have that
\beq
\ip{(\gamma(f)+K)\ \chi_n}{\chi_n}
&=\ &\ip{(f_-(-M_x) +f_+(H)-f(0)+K)\ \chi_n}{\chi_n}\\
&=\ &\ip{f_-(-x)\chi_n}{\chi_n}+\ip{(f_+-f_+(0))(H)\ \chi_n}{\chi_n}+\ip{K\ \chi_n}{\chi_n}\\
&\to\ & \ \ \ \ \ f_-(s)\qquad\qquad +\qquad\qquad\ \  0\qquad\qquad\qquad\  +\qquad 0\\
&=\ &\ \ f_-(s).
\eeq
Therefore, as $\norm{\chi_n}=1$,
\[
|f_-(s)|\le \norm{\gamma(f)+K}
\]
for all $f\in \P_0$ and $K$ any compact operator acting on $\ltwo$.
Hence,
\[
|f_-(s)|\le \inf_{K\in\K_0}\norm{\gamma(f)+K}=\norm{\Gamma(f)}
\]
for all $f\in \P_0$. As $\Gamma$ is continuous and $\P_0$ is dense in $\P$, it follows that \eqref{calk.30} holds for all $f \in \P$.
\end{proof}
\begin{lem}\label{calk.lem.40}
If $z\in \overline{\d(1,1)}$, then
\be\label{calk.40}
|f_+(z)| \le \norm{\Gamma(f)}
\ee
for all $f\in \P$.
\end{lem}
\begin{proof}
We first observe that as $f_+ \in A(\overline{\d(1,1)})$, by the Maximum Modulus Theorem it suffices to prove the lemma under the assumption that $z=1+\tau$ where $\tau \in \t \setminus \{-1\}$. For $\alpha \in \d$, let
\[
\Upsilon_{\alpha} = U^* \frac{k_{-\bar\alpha}}{\norm{k_{-\bar\alpha}}},
\]
where $U$ is as in \eqref{eqb2}.
Clearly, as $k_{-\bar\alpha}/\norm{k_{-\bar\alpha}}$ is a unit vector and $U^*$ is unitary, $\Upsilon_{\alpha}$ is a unit vector. Also, as
\[
(1-S^*)\frac{k_{-\bar\alpha}}{\norm{k_{-\bar\alpha}}}=
(1+\alpha)\frac{k_{-\bar\alpha}}{\norm{k_{-\bar\alpha}}},
\]
it follows that $H\Upsilon_{\alpha}=(1+\alpha)\Upsilon_{\alpha}$, and more generally,
\be\label{calk.50}
f_+(H)\Upsilon_{\alpha} \= f_+(1+\alpha)\Upsilon_{\alpha}
\ee
for all $f\in \P$.

Now notice that \eqref{eqb2} implies that
\[
\Upsilon_{\alpha}\=
\frac{\sqrt{1-|\alpha|^2}}{1+\alpha}x^{-\frac{\alpha}{1+\alpha}}.
\]
\begin{claim}\label{calk.claim.10}
If $\rho>0$ and $\tau \in \t \setminus \{-1\}$, then
\be\label{calk.60}
\lim_{\alpha \to \tau}\ip{x^\rho \Upsilon_{\alpha}}{\Upsilon_{\alpha}} =0.
\ee
\end{claim}
\begin{proof}
First note that
\[
\rho -(\frac{\alpha}{1+\alpha}+\frac{\bar\alpha}{1+\bar\alpha})+1 =\rho+\frac{1-|\alpha|^2}{|1+\alpha|^2},
\]
so that
\[
\int_0^1x^{\rho-(\frac{\alpha}{1+\alpha}+\frac{\bar\alpha}{1+\bar\alpha})} dx
\=\Big(\rho+\frac{1-|\alpha|^2}{|1+\alpha|^2}\Big)^{-1}.
\]
Hence,
\begin{align*}
\ip{x^\rho \Upsilon_{\alpha}}{\Upsilon_{\alpha}}&=\frac{1-|\alpha|^2}{|1+\alpha|^2}
\int_0^1x^{\rho-(\frac{\alpha}{1+\alpha}+\frac{\bar\alpha}{1+\bar\alpha})}dx \\
&=\frac{1-|\alpha|^2}{|1+\alpha|^2}\Big(\rho+\frac{1-|\alpha|^2}{|1+\alpha|^2}\Big)^{-1}\\
&=\frac{1-|\alpha|^2}{|1 +\alpha|^2\rho\ +1-|\alpha|^2}.
\end{align*}
Therefore, if $\rho>0$ and $\tau \in \t \setminus \{-1\}$, \eqref{calk.60} holds.
\end{proof}
Observe that if $q$ is a polynomial and $\tau \in \t \setminus \{-1\}$, then Claim \ref{calk.claim.10} implies that $\ip{q\Upsilon_{\alpha}}{\Upsilon_{\alpha}}\to q(0)$ as $\alpha \to \tau$. In particular,
\be\label{calk.70}
\lim_{\alpha \to \tau}\ip{q(-M_x)\Upsilon_{\alpha}}{\Upsilon_{\alpha}}=0
\ee
whenever $\tau \in \t \setminus \{-1\}$ and $q$ is a polynomial satisfying $q(0)=0$.

We now conclude the proof of the lemma. We need to show that if $f \in \P$ and $\tau \in \t\setminus \{-1\}$ then \eqref{calk.40} holds with $z=1+\tau$. First assume that $f\in \P_0$ and fix $K \in \calk_0$. Since $\Upsilon_{\alpha} \to 0$ weakly as $\alpha \to \tau$, using \eqref{calk.50} and \eqref{calk.60} we have
\beq
\ip{(\gamma(f)+K)\Upsilon_{\alpha}}{\Upsilon_{\alpha}}
&\=&\ip{(f_--f_-(0))(-M_x)\Upsilon_{\alpha}}{\Upsilon_{\alpha}}+
\ip{f_+(H)\Upsilon_{\alpha}}{\Upsilon_{\alpha}}+\ip{K\Upsilon_{\alpha}}{\Upsilon_{\alpha}}\\
&\to &\qquad\qquad\qquad 0 \qquad\qquad\qquad  +\ \ \ \ f_+(1+\tau)\ \ \ \
+\ \ 0\\
&=&f_+(1+\tau).
\eeq
as $\alpha \to \tau$. Therefore, if $f \in \P_0$ and $\tau \in \t\setminus \{-1\}$,
\[
|f_+(1+\tau)| \le \norm{\gamma(f)+K}.
\]
Hence, if $f \in \P_0$,
\[
|f_+(1+\tau)| \le \inf_{K\in K_0}\norm{\gamma(f)+K} = \norm{\Gamma(f)}.
\]
As $\Gamma$ is continuous and $\P_0$ is dense in $\P$, it follows that \eqref{calk.40} holds with $z=1+\tau$ for all $f \in \P$.
\end{proof}
\begin{lem}\label{calk.lem.50}
$\Gamma$ is a homeomorphism.
\end{lem}
\begin{proof}
In the proof of Proposition \ref{calk.prop.20} we showed that
\[
\norm{\Gamma(f)} \le 3\norm{f}
\]
for all $f\in \P$. On the other hand, Lemma \ref{calk.lem.30} implies that
\[
\max_{t\in[-1,0]}|f_-(t)| \le \norm{\Gamma(f)}
\]
for all $f\in \P$ and Lemma \ref{calk.lem.40} implies that
\[
\max_{z\in \overline{\d(1,1)}}|f_+(z)| \le \norm{\Gamma(f)}
\]
for all $f\in \P$. Therefore,
\[
\norm{f}=\max\ \big\{\max_{t\in[-1,0]}|f_-(t)|, \max_{z\in \overline{\d(1,1)}}|f_+(z)|\ \big\} \le \norm{\Gamma(f)}
\]
for all $f\in \P$.
\end{proof}
Putting together the results of Subsection \ref{secc2} we get the following theorem. 
\bt
\label{thmc1}
The map $\Gamma$ is a homeomorphic unital isomorphism from $\P$ onto $\calc$. 
\et

\subsection{Some Observations on the Gelfand Theory of $\calc$}
\label{secc3}

If $\L=[-1,0]\cup \overline{\d(1,1)}$, then there is an isometric isomorphism from
$\P$ onto the lollipop algebra $A(\L)$ given by
\[
A(\L) \ni f \mapsto \big(\ f|_{[-1,0]}\ , f|_{(\overline{\d(1,1)})}\ \big).
\]
 So one could just as well state Theorem \ref{thmc1} with $\P$ replaced by $A(\L)$ and $\Gamma$ replaced with the map $\Gamma^\sim:A(\L) \to \calc$ defined by
\[
\
\Gamma^\sim (f)=\Big[\ f \big|_{[-1,0]}(-M_x) + f\big|_{\overline{\d(1,1)}}(H) -f(0)\ \Big].
\]
\begin{defin}
Define $\theta: \A \to A(\L)$ by 
\[
\theta(X) \= (\Gamma^\sim)^{-1} ( [X] ) .
\]
\end{defin}
Then Theorem \ref{thmc1} says that there is a short exact sequence
 \[
0 \to\Kz  \to \A 
\stackrel{\theta}{ \to}
  A(\L)  \to 0 .
\]

\begin{rem}
\label{remc1}
By Mergelyan's theorem, $A(\L) = P(\L)$, and  since $z$ generates $P(\L)$, it follows that
\[
\Gamma^\sim(z)=[H-M_x]
\]
generates $\calc$.
We shall examine $H - M_x$ in Section \ref{secg}.
\end{rem}


\section{$C^*(\A)$}
\label{secd}

We shall let $\BB = C^*(\A)$ denote the $C^*$-algebra generated by $\A$. Since it is irreducible, 
$\BB$ contains all the compact operators.

The Toeplitz $C^*$-algebra $\mathcal T$ is the $C^*$ algebra generated by the shift $S$.
There is a  short exact sequence
\[
0 \to\K  \to {\mathcal T}
\stackrel{\alpha}{ \to}
\  A(\overline{\D})  \to 0 .
\]
(See e.g. \cite[7.23]{dou72}). A cross-section of $\alpha
$ is the map that sends a function $m$ to the Toeplitz 
operator $T_m$ on $H^2$ with symbol $m$.

Since $\widehat{H} = S^* + 1$, the $C^*$-algebra generated by $H$ is unitarily equivalent to $\mathcal T$.
We wish to think of it as living on $\overline{\D(1,1)}$, so we must shift things over.
Let $\tau(z) = z+1$. For any function $f$ defined on some domain in $\C$, let $f^\cup(z) = \overline{f(\bar z)}$ be its reflection in the real axis.

\begin{definition}
Let  $\psi \in C( \partial \D(1,1))$.
Let $H_\psi \in B(L^2)$ be defined by
\[
H_\psi  \=  U^*T^*_{(\psi \circ \tau)^\cup} U .
 \]
\end{definition}

The map $ \psi \mapsto H_\psi$ is unital and linear, but not multiplicative. 
One checks that if $\psi(z) = z^n$, then $H_{z^n}  = H^n$, and if $\psi(z) = \bar z^n$, then 
$H_{\bar z^n} = (H^*)^n$.

\bt
\label{thmd1}
There is a short exact sequence
\be
\label{eqha0}
0 \to \K \to \BB \stackrel{\pi}{ \to} C(\partial \L) \to 0 .
\ee
For every $X$ in $\BB$, its coset in $\BB/\K$  can be written uniquely as
\be
\label{eqha1}
[X ] \= [g(-M_x) + H_\psi  -g(0) ]
\ee
where $g \in C[-1,0]$,  $\psi \in C( \partial \D(1,1))$, and $g(0) = \psi(0) $.
The essential spectrum of $X$ as in \eqref{eqha1} is $g([0,1]) \cup \psi ( \partial \D(1,1))$.
If $\l \notin \sigma_e(X)$, then the Fredhom index is given by the winding number of $\psi$ about $\l$:
\[
{\rm ind}(X - \l) \= {\rm ind}_{\psi}(\l) .
\]
\et

\bp
For $X \in \BB$, we shall let $[X]$ denote its equivalence class in $\BB/\K$.
We have
\[
[ H M_x ] = [ M_x H] = 0 ,
\]
and
\[
[ HH^* ] = [H^*H] = H + H^* .
\]
So $\BB/\K$ is abelian. Moreover, for any polynomial $q$ in 3 variables, there are polynomials 
$p_1, p_2, p_3$ in one variable so that
\be
\label{eqd1}
[q(M_x, H, H^*) ] \= [p_1(M_x) + p_2 (H) + p_3(H^*)].
\ee
Therefore operators of the form \eqref{eqd1} are dense in $\BB/\K$.

We wish to prove that $\BB/\K$ is isomorphic to the abelian $C^*$-algebra $C(\partial \L)$.
We will use a similar strategy to the proof of Theorem \ref{thmc1}.
Let 
\[
\Q \= \{ f = (f_-, f_+) : f_- \in C([-1,0]), f_+ \in C(\partial \D(1,1)), f_-(0) = f_+(0) \} .
\]
The algebra $\Q$ is just $C(\partial \L)$, but it is easier to define the functional calculus on it.
Define 
\beq
\delta: \Q &\ \to\ & \BB \\
f &\mapsto & f_-(-M_x) + H_{f_+} - f(0) .
\eeq
\ep
Let $\Delta(f) = [ \delta(f) ] $.
The following lemma is straightforward to prove.
\begin{lem}
\label{lemd1}
(i) Let $\psi, \phi \in \cpd$. Then $[ H_\psi H_\phi] = [H_{\psi \phi} ]$.

(ii) Let $ g \in C([-1,0])$ and $\psi \in \cpd$. Then
\[
[g(-M_x) ] [ H_\psi]\=  [ H_\psi] [g(-M_x) ]   \= [g(0) H_\psi + \psi(0) g(-M_x) - g(0) \psi(0)] .
\]
\end{lem}
Using Lemma \ref{lemd1}, one can check that $\Delta$ is a unital *-homomorphism from
$\Q$ into $\BB/\K$. Its range is dense, so if we can show it has no kernel, then it is a C*-isomorphism.

\begin{lem}\label{lemd2}
If $s\in [-1,0]$, then
\be\label{eqd3}
|f_-(s)| \le \norm{\Delta(f)}
\ee
for all $f\in \Q$.
\end{lem}
\begin{proof}
As $f$ is continuous, it suffices to prove the lemma under the assumption that $s\in (-1,0)$,
and as $\Delta$ is continuous, 
we can assume that $f_-$ is a polynomial, and that $f_+(z) = f(0) + z p_2(z) + \overline{ zp_3(z)}$
where $p_2$ and $p_3$ are polynomials.

As in Lemma \ref{calk.lem.30}, 
for $n$ satisfying $1/n < \min \{s, 1-s\}$ we define a unit vector $\chi_n \in \ltwo$ by the formula
\[
\chi_n(t)=
\twopartdef{\sqrt{\frac{n}{2}}}{|t-s|\le 1/n}{0}{|t-s|>1/n}
\]
Let $K$ be compact.

\beq
\ip{(\delta(f)+K)\ \chi_n}{\chi_n}
&=\ &\ip{(f_-(-M_x) + H_{z p_2} + H_{\bar z \bar p_3} +K)\ \chi_n}{\chi_n}\\
&=\ &\ip{f_-(-x)\chi_n}{\chi_n}+\ip{p_2(H) H\ \chi_n}{\chi_n}+\ip{\chi_n}{p_3(H) H\ \chi_n}+  \ip{K\ \chi_n}{\chi_n}\\
&\to\ & \ \ \  f_-(s)\qquad\quad\  +\qquad\qquad  0\qquad\ \ \    +\qquad\  0 \qquad\ \ \ \ \ \  \ + \qquad 0\\
&=\ &\ \ f_-(s).
\eeq
Therefore, 
\[
|f_-(s)|\le \inf_{K\in\K}\norm{\delta(f)+K}=\norm{\Delta(f)}.
\]
\end{proof}

If $\Delta(f) = 0$, by Lemma \ref{lemd2} we must have $f_- = 0$. 
So $\delta(f) = \hfp$ must be compact. But $\hfp$ is unitarily equivalent to a Toeplitz operator, and there are
no non-zero compact Toeplitz operators.
Therefore $\Delta$ has a trivial kernel, and hence is a $*$-isomorphism.

The claim about the spectrum of $[x]$ now follows from the fact that the spectrum of a function in $C(\partial \L)$
equals its range. Finally, the claim about the Fredholm
index follows from the fact that the Fredholm index at $\lambda$ will be unchanged under any homotopy
of $f$  that keeps $\lambda$ outside its range. Then $f$ can be homotoped to $(f_-, f_+)$ where
$f_+(z) = (z-1-\lambda)^n$ and $f_-(x) = (-1-\lambda)^n$ for some integer $n$, and the Fredholm index of $\delta(f) $ is $n$.
\ep

%
%

\section{Compact operators in the little algebra $\Az$}
\label{secl}

Recall from Definition \ref{defina3} that $\Az$ is the norm-closed algebra generated by $M_x$ and $V$.
We shall prove that every compact operator in ${\rm AlgLat}(V)$ lies not just in $\A$ but in $\A_0$.

For $I$ an interval in $[0,1]$, let  us write $L^2(I)$ for the subspace of $L^2$ that vanishes a.e.
off $I$, and let 
$P_I$ denote projection onto $L^2(I)$. For $\phi, \psi \in L^2$ we write $\phi \otimes \psi$ to denote the rank one operator
\[
\phi \otimes \psi : f \ \mapsto \  \la f, \psi\ra \phi .
\]

The key observation is the following:
\begin{lemma}
\label{leml0}
Suppose $\phi \in L^2[t,1]$ and $\psi \in L^2[0,t]$. Then $\phi \otimes \psi = M_\phi V M_\psi^*$.
\end{lemma}
\bp
We have
\[ M_\phi V M_\psi^* f (x) \= \phi(x) \int_0^x f(s) \overline{\psi (s)} ds .
\]
The right-hand side is $0$ if $x < t$, and $\phi(x) \la f , \psi \ra$ if $x > t$.
\ep

\begin{lemma}
\label{leml3}
Every finite rank operator on $L^2[0,1]$ can be written as an integral operator whose kernel is in $L^2([0,1] \times [0,1])$.
\end{lemma}
\bp
Let $K = \sum_{j=1}^n \phi_j \otimes \psi_j$. Define 
\[
k(x,s) \= \sum_{j=1}^n \phi_j(x) \overline{\psi_j(s)} .
\]
Then
$K f (x) = \int_0^1 k(x,s) f(s) ds$, and $k$ is in $L^2([0,1] \times [0,1])$.
\ep

\begin{lemma}
\label{leml1}
Let $k$ be in  $L^2([0,1] \times [0,1])$, and let
$T f (x) = \int_0^1 k(x,s) f(s) ds$. Then $T$ is in ${\rm AlgLat}(V)$ if and only if
$k(s,x) = 0$ for $s > x$.
\end{lemma}
\bp
Sufficiency is clear. To prove necessity, assume that for
 some $ 0 < t < 1$, the kernel
\[
k(s,x) \chi_{[0,t](x)} \chi_{[t,1]}(s) 
\]
is not $0$ a.e. As an integral operator is zero if and only if  the kernel is $0$ a.e.,
this means that the corresponding integral operator is non-zero, and hence $T$ maps a 
function in $L^2 (t,1)$ to a function that is not $0$ a.e. on $[0,t]$.
\ep

\begin{lemma}
\label{leml2}
 Let $T = \phi \otimes \psi$ be a rank-one operator. Then $T$ is in ${\rm AlgLat}(V)$
if and only if for some $ 0 < t < 1$, the support of $\phi$ is in $[t,1]$ (i.e. $\phi = 0 $ a.e. on $[0,t]$) and
the support of $\psi$ is in $[0,t]$. In this case, $T \in \A_0$.
\end{lemma}
\bp
The first part follows from Lemma \ref{leml1}. For the second part, observe that
if the supports of $\phi, \psi$ are in $[t,1]$ and $[0,t]$ respectively, then 
$\phi \otimes \psi = M_\phi V M_\psi^*$. If $\phi$ and $\psi$ are both in $C([0,1])$, this proves
that $\phi \otimes \psi \in \A_0$. 

For the general case, choose continuous functions  $f_n$ and $g_n$ that converge to $\phi$ and $\psi$ respectively in $L^2$.
It  follows from Lemma \ref{leml0} that $M_{f_n} V M_{g_m}^*$ converges to $M_\phi V M_{g_m}^*$ in norm as $n \to \infty$, and that $M_\phi V M_{g_m}^*$ converges to $ M_\phi V M_\psi^*$. Therefore  $ \phi \otimes \psi \in
\A_0$.
\ep

\bt
\label{thml1}
Let $K$ be a compact operator in ${\rm AlgLat}(V)$. Then $K \in \A_0$,
and can be approximated in norm by finite rank operators in $\A_0$.
\et
\bp
Note that $K \in {\rm AlgLat}(V)$ means that for all $ 0 < s < 1$, we have
$P_{[0,s]} K P_{[s,1]} = 0$.

Let $\vare > 0$. First, consider $P_{[1/2,1]} K P_{[0,1/2]}$. This can be approximated within $\vare/2$ by a finite rank 
operator that is a sum of rank one operators that map $L^2(0,1/2)$ to $L^2 (1/2,1)$.
By Lemma \ref{leml2}, this means that this finite rank operator is in $\A_0$.

A similar argument shows that $P_{[1/4, 1/2]} K P_{[0,1/4]}$ and $P_{[3/4,1]} K P_{[1/2,3/4]}$
can both be approximated by finite rank operators in $\A_0$ within $\vare/8$.
Iterating, we get that if $n$ is a power of $2$, we can approximate
\[
K - \sum_{j=1}^n P_{[(j-1)/n, j/n]} K P_{[(j-1)/n, j/n]}
\]
within $\vare$ by a finite rank operator in $\A_0$.

Finally we observe that
\be
\label{eql1}
\| \sum_{j=1}^n P_{[(j-1)/n, j/n]} K P_{[(j-1)/n, j/n]} \| 
\= \max_{1 \leq j \leq n} \| P_{[(j-1)/n, j/n]} K P_{[(j-1)/n, j/n]} \|.
\ee
Since $K$ is compact, 
\[
\lim_{ n \to \i} \sup_{|I| = 1/n} \| K P_I \| = 0 ,
\]
so \eqref{eql1} tends to $0$.
\ep

\section{The operator $H-M_x $}
\label{secg}

Let us write $Z$ for the operator $H-M_x$.
We know that $[Z]$ generates the Calkin Hardy-Weyl algebra $\calc$.
By Theorem \ref{thmc1} we know that the spectrum of $[Z]$ in $\A /\K_0$ is $\L$.
It is not surprising that $\D(1,1)$ are eigenvalues of $Z$, since they are eigenvalues of $H$.
It is perhaps surprising that every point in the stick, except $-1$, is also an eigenvalue.
Moreover as we move up the stick to the bulb of the lollipop, the eigenvalues increase in multiplicity.
\bt
\label{thmg1}
(i) $\sigma_p (Z) = \D(1,1) \cup (-1,0]$.

(ii) The point spectrum of $Z^*$ is empty.

(iii) The spectrum of $Z$ is $\L$.
\et

\bp 
 (i) Suppose $(Z - \lambda ) f = 0$.
Let $F(x) = Vf(x) = \int_0^x f(t) dt$.
Then we have
\[
\frac{1}{x}F(x) \= (x+\l) f(x) .
\]
As $F'(x) = f(x)$, we get the equation
\be
\label{eqg1}
\frac{1}{x}F(x) = (x+\l) F'(x),
\ee
with the boundary condition 
\be
\label{eqg15}
 \quad F(0) \= 0.
 \ee
The function $F$ is continuous. Let $\Omega$ denote the relatively open subset of $[0,1]$ on which it is non-zero.

We get that the solution of 
\eqref{eqg1}, with $\l \neq 0$, is
\be
\label{eqg14}
F(x) \= c \left( \frac{x}{x+\l} \right)^{1/\l}  \chi_\Omega(x)\ ,\ee
where the constant $c$ can a priori be different on different components of $\Omega$ (though we show
below that $\Omega$ is actually connected). Hence
\be
\label{eqg2}
f(x) \= c x^{\frac{1}{\l} - 1} (x+\l)^{-1 - \frac{1}{\l}}  \chi_{\Omega}(x).
\ee
Case: $\Omega = [0,1]$.

For $f$ to be in $L^2$ with $c \neq 0$, we need 
\[
\Re (\frac{1}{\l} - 1) \ > \ - \frac{1}{2} ,
\]
which is the same as $\l \in \D(1,1)$.
For $\l \in \D(1,1)$, we get the eigenfunctions
\be
\label
{eqg16}
f(x) \=  x^{\frac{1}{\l} - 1} (x+\l)^{-1 - \frac{1}{\l}} .
\ee
When $\l = 0$, we get
\[ F(x) \= c e^{-1/x}, \]
and
\be
\label{eqg17}
f(x) \= c\frac{1}{x^2} e^{-\frac{1}{x}} .
\ee

Now suppose that $\Omega$ is not all of $[0,1]$. Decompose $\Omega$ as a union of disjoint non-empty intervals. On each interval, we have that $f$ is given by \eqref{eqg2}, with some 
 constant $c$ that  can depend on the interval.
Choose such an interval, $I$. Then $0$ cannot be an end-point of $I$, or we would have that $F$ is given by \eqref{eqg14}, and this is discontinuous at the right-hand end-point of $I$.

So assume that the left-hand end-point of $I$ is $t > 0$. Then the boundary condition \eqref{eqg15} is replaced by $F(t) = 0$.
On $I$, we have
\[
F(x) \= c \left( \frac{x}{x+\l} \right)^{1/\l}  ,
\]
so to have $F(t) = 0$ we need $\l = - t$. By continuity, $F$ cannot vanish again, so we conclude that $I = (t,1]$ and that for
$\l \in (-1,0)$ 
the function
\be
\label{eqg3}
f(x) \= x^{\frac{1}{\l} - 1} (x+\l)^{-1 - \frac{1}{\l}} \chi_{[-\lambda,1]}(x)
\ee
is an eigenvector of $Z$ with eigenvalue $\l$.

(ii) As $H^* f(x) = \int_x^1 \frac{1}{t} f(t) dt$, the eigenvalue equation becomes
\[
G(x) \= (x+\l) f(x) , 
\]
where
\[
G(x) \=  \int_x^1 \frac{1}{t} f(t) dt \]
As $G'(x) = - \frac{1}{x} f(x)$, we get the differential equation
\be
\label{eqg4}
G(x) \= - x (x + \l)  G'(x) , \quad G(1) = 0 .
\ee
Solving for $G$, we get
\[
G(x) \= c x^{-\frac{1}{\l}} (x+\l)^{\frac{1}{\l} } \, \chi_{\{ G \neq 0 \}} .
\]
The only solution on an interval that vanishes on the right end-point is the zero solution.

(iii) We know \[
\sigma_e( Z) = \partial \L  \subseteq \L = \overline{\sigma_p(Z)}  \subseteq \sigma(Z) .
\]
If $\l$ is a point in $ \C \setminus \L$, it must be a Fredholm point.
By (i), it is not an eigenvalue of $Z$, and by (ii) it is not an eigenvalue of $Z^*$.
Therefore $Z - \l$ has trivial kernel and cokernel, and closed range. Therefore it is invertible,
and $\l$ is in the resolvent of $Z$.
\ep

Not only are points on the stick of the lollipop eigenvalues, there is some additional smoothness.
 By a generalized eigenvector of order $n$ we
mean a vector $f$ that satisfies $(Z-\lambda)^{n+1} f = 0$ but $(Z-\lambda)^n f \neq 0$.

We shall prove the case $\lambda = 0$ first.

\begin{lem}
\label{lemg1}
At $0$, the operator $Z$ has generalized eigenvectors of all orders.
\end{lem}
\bp
We want to show that if we let $f_0(z) = \frac{1}{x^2} e^{-\frac{1}{x}}$ from
\eqref{eqg17}, then for every $n \in \NN$ there exists $f \in L^2$ so that
\be
\label{eqg21}
Z f_{n+1} \= f_n .
\ee
\begin{claim}
\label{clg2}
For every $n \in \NN$ there exists a polynomial $p_n$ of degree $2n+2$, with lowest order term
of degree $n+2$, so that the functions 
\[
f_n(x) \= p_n(\frac{1}{x} ) e^{- \frac{1}{x}}
\]
satisfy \eqref{eqg21}.
\end{claim}
We prove this by induction on $n$. It is true when $n=0$. Assume we have proved it up to level $n$,
and we want to prove it for $n+1$.
So we wish to solve the equation
\be
\label{eqg22}
Z f_{n+1} \= f_n = p_n(\frac{1}{x} ) e^{- \frac{1}{x}}
\ee
and show that the solution is of the form
\be
\label{eqg23}
f_{n+1}(x) \= p_{n+1}(\frac{1}{x} ) e^{- \frac{1}{x}}.
\ee
Writing $\Fno$ for $V \fno$, equation \eqref{eqg22} is
\beq
(H- M_x) \fno (x) &\=& \frac{1}{x} \int_0^x \fno(t) dt - x \fno(x) \\
&=& 
\frac{1}{x}  \Fno(x)  - x \Fno'(x) \\
&=& 
f_n(x).
\eeq
This gives us the linear differential equation
\[
\Fno'(x) - \frac{1}{x^2} \Fno(x) \= - \frac{1}{x} f_n(x) .
\]
Multiply by the integrating factor $e^{\frac{1}{x}}$ to get
\beq
\frac{d}{dx} \left[ e^{\frac{1}{x}} \Fno(x) \right] &\=& 
 - \frac{1}{x} e^{\frac{1}{x}} f_n(x) \\
 &=& 
 - \frac{1}{x} p_n(\frac{1}{x}) .
 \eeq
 Therefore
 \[
  e^{\frac{1}{x}} \Fno(x) \= q_{n} ( \frac{1}{x}) ,
  \]
  where $q_n$ is a polynomial of degree $2n+2$ that may have a constant term, and whose next lowest order term is of degree $n+2$.
  This gives
  \beq
  \fno(x) & \=&  \frac{d}{dx} \left[ e^{-\frac{1}{x}}q_{n} ( \frac{1}{x}) \right]
  \\
  &=& 
e^{-\frac{1}{x}}\left[  \frac{1}{x^2} q_n (   \frac{1}{x})  -  \frac{1}{x^2} q_n' (   \frac{1}{x})\right]
\eeq
Let 
\[
p_{n+1} (x) \= x^2 q_n(x) - x^2 q_n'(x) .
\]
The degree of $p_{n+1}$ is two higher than $q_n$, so it is  $2(n+1) + 2$. There may be a term of order $2$;
the next lowest order term is $n+3 = (n+1) + 2$.
But as $Z f_0 = 0$, one can subtract a multiple of $f_0$ from $f_{n+1} $ without changing
\eqref{eqg22}, so we can assume that $p_{n+1}$ has no term of order $2$.

So we have proved Claim \ref{clg2}. As any function of the form \eqref{eqg23} is in $L^2$,
we are done.
\ep

For points in the stick, a similar method works, but there are restrictions when requiring the
generalized eigenvectors to be in $L^2$. Here  is one result.
\begin{lem}
\label{lemg2}
Let $\lambda \in (-1,0)$. Then $Z$ has a generalized eigenvalue of order $1$ at $\lambda$ if and only if
$-\frac{2}{3} < \lambda < 0$.
\end{lem}
\bp
Let $s = -\lambda$. 
\be
\label{eq29}
\fz(x) \=  \left( \frac{x-s}{x} \right)^{\frac{1}{s}} \frac{1}{x(x-s)} \chi_{[s,1]} .
\ee
All the functions below are supported on $[s,1]$.
We wish to find a function $\fo$ that satisfies
\be
\label{eqg30}
(Z + s) \fo \= \fz .
\ee
Writing $\Fo$ for $V \fo (x)= \int_s^x \fo(t) dt$, this becomes
\[
\frac1x \Fo - (x-s) \Fo' \= \fz.
\]
or
\be
\label{eqg31}
\Fo' - \frac{1}{x(x-s)} \Fo \= - \frac{1}{x-s} \fz .
\ee
An integrating factor for \eqref{eqg31} is $\left( \frac{x}{x-s} \right)^{\frac{1}{s}}$.
This yields
\beq
\frac{d}{dx} \left[  \left( \frac{x}{x-s} \right)^{\frac{1}{s}}\Fo \right]
&\=&
- \left( \frac{x}{x-s} \right)^{\frac{1}{s}} \frac{1}{x-s} \fz
\\
&=&
- \frac{1}{x(x-s)^2} .
\eeq
Integrating, we get
\[
 \left( \frac{x}{x-s} \right)^{\frac{1}{s}}\Fo \=
 \frac{1}{s^2} \log \frac{x-s}{x} + \frac1s \frac{1}{x-s} + c .
 \]
Dividing through by the integrating factor and differentiating, we get
\beq
\fo(x) &\=& 
\frac{1}{s^2} \left[   \left( \frac{x-s}{x} \right)^{\frac{1}{s}-1} \frac{1}{x^2} \log \frac{x-s}{x} 
+  \left( \frac{x-s}{x} \right)^{\frac{1}{s}} \frac {s}{x(x-s)} \right] \\
&& + 
\frac{1}{s} \left[ -\frac{1}{(x-s)^2} \left( \frac{x-s}{x} \right)^{\frac{1}{s}} + \frac{1}{x-s} \left( \frac{x-s}{x} \right)^{\frac{1}{s}-1} \frac{1}{x^2} \right] \\
&&+ c \left[ \left( \frac{x-s}{x} \right)^{\frac{1}{s}-1} \frac{1}{x^2} \right] .
\eeq
We can choose $c=0$, since it is the coefficient of $\fz$. This gives
\be
\label{eqg26}
\fo(x) \= \left( \frac{x-s}{x} \right)^{\frac{1}{s}} \left[
\frac{1}{s^2} \frac{1}{x(x-s)}  \log \frac{x-s}{x} 
+ \frac{1-s}{s} \frac{1}{x(x-s)^2} \right] .
\ee
Examining this expression, we see that $\fo$ is smooth on $(s,1]$, and the first  term
\[
\frac{1}{s^2}
 \frac{(x-s)^{\frac{1}{s} - 1}}{x^{\frac{1}{s} + 1}}
  \log \frac{x-s}{x} 
\]
vanishes  at $s$ for every $s < 1$. However the second term
\[
\frac{1-s}{s}\  \frac{(x-s)^{\frac{1}{s} - 2}}{x^{\frac{1}{s} +1}} 
\]
 has a singularity that grows like
$(x-s)^{\frac{1}{s} - 2} $.
This is integrable for every $s < 1$, but it is only in $L^2[s,1]$ for $s < \frac{2}{3}$.
So we have shown that \eqref{eqg30} has a solution $\fo$ in $L^2$ if and only if
$\lambda > - \frac23$.
\ep

One can repeat the argument of Lemma \ref{lemg2} to get higher order generalized eigenvectors,
as $\lambda$ gets closer to $0$.

\begin{lem}
\label{lemg3}
Let $m \geq 1$.
Let $\lambda$ lie in the interval $(- \frac{2}{2m+1}, 0)$. Then $Z$ has generalized eigenvectors up to order $m$ at $\l$.
\end{lem}
\bp
We shall inductively find functions $f_n$ satisfying $(Z-\l) \fno = f_n$, with $\fz$ as in \eqref{eq29}.
Let $s = -\l$, and 
 write
\[
\Phi(x) \=   \left( \frac{x-s}{x} \right)^{\frac{1}{s}} \chi_{[s,1]}(x) .
\]
Then we have
\beq
\fz(x) &\=& \Phi(x) \frac{1}{x(x-s)}  \\
\fo(x) &\=&  \Phi(x) \left[
\frac{1}{s^2} \frac{1}{x(x-s)}  \log \frac{x-s}{x} 
+ \frac{1-s}{s} \frac{1}{x(x-s)^2} \right] .
\eeq
Writing $\Fno$ for $V \fno = \int_s^x \fno(t) dt$, we want to solve
\be
\label{eqg41}
\frac{1}{x} \Fno(x) - (x-s) \Fno'(x) \= f_n(x) .
\ee
After multiplying by the integrating factor $1/\Phi$, we have
\be
\label{eqg33}
\frac{d}{dx} \left[ \frac{1}{\Phi} \Fno \right] \= - \frac{1}{\Phi} \frac{1}{x-s} f_n (x) .
\ee
\begin{claim}
\label{clg31} There are constants $M_n$ such that the functions $f_n$  satisfy
\be
\label{eqg42}
|f_n(x) | \ \leq \ M_n (x-s)^{\frac1s - n - 1} \qquad \forall x \in (s,1] .
\ee
\end{claim}
{\sc Proof of Claim \ref{clg31}:}
By induction on $n$. It is true when $n=0$. Assume it is true up to $n$.
From \eqref{eqg33} we get for $x \in (s,1]$:
\be
\label{eqg43}
\frac{1}{\Phi(x) } \Fno (x) \= \int_x^1 \frac{1}{\Phi (t)} \frac{1}{t-s} f_n (t) dt + c_n .
\ee
By the inductive hypothesis, the integrand in \eqref{eqg43} is $O(t-s)^{-n-2}$, so the integral is
$O(x-s)^{-n-1}$.
So $\Fno$ satisfies
\be
\label{eqg44}
\Fno(x) \= c_n \Phi(x) + O(x-s)^{\frac1s-n-1} .
\ee
From \eqref{eqg41} we have 
\be
\label{eqg45}
\fno (x) \= \frac{1}{x(x-s)} \Fno (x) - \frac{1}{x-s} f_n(x) .
\ee
When we use \eqref{eqg44} for $\Fno$,
we get
\be
\label{eqg46}
\notag
\fno (x) \= c_n f_0(x) - \frac{1}{x-s} f_n(x) +O(x-s)^{\frac1s-n -2}.
\ee
Now
the claim follows from  the inductive hypothesis on $f_n$. \ep
It follows from \eqref{eqg45} that  that $f_m$ is continuous on $(s,1]$, and Claim \ref{clg31}  shows that its singularity at $s$ is of order
$(x-s)^{\frac{1}{s} -m -1 }$. This means $f_m$ is in $L^2$ provided
$
\frac{1}{s} -m -1 \ > \ - \frac{1}{2} , 
$
which is the same as $ s < \frac{2}{2m+1}$.
\ep
For later use, let us note that if you track the constants $M_n$ in Claim \ref{clg31}, you can show:
\begin{lem}
\label{lemg35}
In Claim \ref{clg31} one can take $M_0 = \frac{1}{(s^{1+1/s})}$ and the consants $M_n$ satisfy
\[
M_{n+1} \leq M_n \left( 1 + \frac{M_0}{n+1}\right) .
\]
\end{lem}
\begin{lemma}
\label{lemg4} Let $m \geq 1$.
On the interval $(-\frac{2}{2m+1}, 0)$ one can choose the generalized eigenvectors of order $n$ 
of $Z - \l$ continuously in $\l$, for every $n \leq m$, and satisfying $(Z- \l) f_{\l,n} = f_{\l,n-1}$,
for every $1 \leq n \leq m$.
\end{lemma}
\bp
Let us write $f_{\l,n}$ for the choice of generalized eigenvector of order $n$ at $\l$.
Write
\[
\Phi(\l,x) \= \left( \frac{x+\l}{x} \right)^{-\frac{1}{\l}} \chi_{[-\l,1]}(x) .
\]
We have
\[
f_{\l,0} (x) \= \frac{1}{x(x+\l)} \Phi(\l,x) .
\]
On every compact subset $K$, of $(-1,0)$ the functions $\{ f_{\l,0} : \l \in K \}$ are uniformly bounded,
and $\lim_{\l' \to \l} f_{\l',0}(x) = f_{\l,0} (x)$ a.e., 
so the map $\l \mapsto f_{\l,0}$ is continuous as a map from $(-\frac{2}{2m+1}, 0)$ into $L^2$.

For higher $n$, we find $\fno$ as in Lemma \ref{lemg3}. At each stage, we take the constant $c_n$ in \eqref{eqg44} to be $0$. (We can do this because $c_n f_{\l,0}$ will be in the kernel of $Z-\lambda$.)
This gives us
\be
\notag
f_{\l,n+1} (x) \= \frac{1}{x(x-s)} \Phi(\l,x)   \int_x^1 \frac{1}{\Phi (\l,t)} \frac{1}{t-s} f_{\l,n} (t) dt- \frac{1}{x-s} f_{\l,n}(x) .
\ee
Moreover we have $f_{\l,n}(x) = 0$ if $x < - \l$ and 
\[
|f_{\l,n}(x)| \ \leq \ M_{n, \l} (x+\l)^{- \frac{1}{\l} - n - 1}, \qquad x > - \l .
\]
By Lemma \ref{lemg35}, we have that
 each $M_{n,\l}$ can be chosen uniformly in $\l$  
for $\l$ in $(- \frac{2}{2m+1}, 0 )$.
For any interval $I$ of length $\delta$, we get
\beq 
\int_I |f_{\l,n}(x)|^2 dx &\ \leq \ &M_n \int_0^\delta |x+\l|^{- \frac{2}{\l} - 2n - 2}\\
&=& M_n \frac{1}{- \frac{2}{\l} - 2n - 1} \delta^{- \frac{2}{\l} - 2n - 1} .
\eeq
Therefore, as $\l$ ranges over any  compact subset of $(-\frac{2}{2m+1}, 0)$, the functions $|f_{\l,n} (x) |^2$ are uniformly
integrable in $x$. So, by the Vitali convergence theorem, the map 
\beq
(-\frac{2}{2m+1}, 0) & \ \to \ & L^2 \\
\l & \mapsto & f_{\l,n}
\eeq
is continuous.
\ep

These lemmas say that operators in the closed algebra generated by $Z$ have certain smoothness properties when mapped by $\sigma$ into
$A(\L)$.
The functions get smoother as we get closer to $0$.

\begin{thm}
\label{thmg2}
Let $X$ be in the norm-closed algebra generated by $Z$. Then $\theta(X)$ is  $C^m$ 
  on $ (-\frac{2}{2m+1},0)$.
\end{thm}
\bp
Let  $\theta(X) = \phi \in A(\L)$.
Let  $p_j$ be a sequence of polynomials  so that $\| p_j(Z) - X \| \to 0$.
It follows from Theorem \ref{thmc1} that $p_j$ converges to $\phi$ uniformly on $\L$.

Case: $m=1$.
Let $f_{\l,n}$ be as in Lemma \ref{lemg4}.
For any polynomial $p$, we have
\[
\la p(Z) f_{\l,0} , f_{\l,0}\ra \= \| f_{\l,0} \|^2 p(\l) .
\]
Moreover, 
\[
p(Z) f_{\l,1} \= p(\l) f_{\l,1} + p'(\l) f_{\l,0} .
\]
Let $g^1_\l$ be the linear combination of $f_{\l,0}$ and $f_{\l,1}$ that satisfies $\la f_{\l,0} , g^1_\l \ra = 1$ and $\la f_{\l,1} , g^1_\l \ra = 0$.
Then
\be
\notag
\la p(Z)  f_{\l,1}, g^1_\l \ra \= p'(\l) .
\ee
So as functions on $ (-\frac23, 0)$, we get that $p_j'(\l)$ converges to some function $\psi(\l) =
\la X  f_{\l,1}, g^1_\l \ra$.

\begin{claim}
For all $x$ in  $(-\frac23, 0)$, we have
$\phi'(x) = \psi(x)$, and $\psi$ is continuous.
\end{claim}
We have $g^1_\l = a_{\l,0} f_{\l,0} + a_{\l,1} f_{\l,1}$,  
where the coefficients $a_{\l,0} $ and $a_{\l,1}$ solve the linear system
\[
\begin{pmatrix} 
\la f_{\l,0} ,  f_{\l,0}\ra & \la  f_{\l,0},  f_{\l,1} \ra \\
\la  f_{\l,1},  f_{\l,0} \ra & \la  f_{\l,1},  f_{\l,1} \ra 
\end{pmatrix}
\begin{pmatrix}
a_{\l,0} \\
a_{\l,1} 
\end{pmatrix}
\=
\begin{pmatrix}
1\\0
\end{pmatrix}.
\]
By Lemma \ref{lemg4}, since $f_{\l,0}$ and $f_{\l,1}$ are continuous in $\l$, so
as they are linearly independent,
we have that $g^1_\l$ is also continuous in $\l$.
Therefore $\psi$ is continuous, and 
 $p_j'$ converges to $\psi$ locally uniformly on $(-\frac23,0)$.
 As $\int_{-\frac13}^x p_j'(t) dt = \phi(x) - \phi (- \frac13)$ converges to $\int_{-\frac13}^x \psi(t) dt$, we get that
 $\phi'(x) = \psi(x)$.
\ep
We have shown that $\phi$ is in $C^1(-\frac23,0)$. A similar argument with higher derivatives proves that $\phi$ is in $C^m(-\frac{2}{2m+1},0)$.
\ep
Of course one can also find generalized eigenvectors for $Z$ of all orders at points in $\D(1,1)$,
but we already know that $\theta(X)$ is analytic on $\D(1,1)$ for every $X \in \A$.

\section{Open Questions}

\begin{question}
Is there a good description of $\Kz$, the compact operators in $\A$?
\end{question}

\begin{question}
Is $C^*(Z) = C^*(\A)$? To prove this, it is sufficient to show that $C^*(Z)$ is irreducible, since
then it  would contain all the compacts, and its quotient by the compacts would be all of $C(\partial \Lambda)$.
\end{question}

\begin{question}
Do the eigenvectors of $Z$ span $L^2$?
\end{question}


\bibliography{../../references_uniform_partial}
\end{document}